&pdflatex
\documentclass{amsart}
\usepackage{amsmath} 
\usepackage{amssymb} 
\usepackage{graphicx}
\usepackage{epsfig} 
\usepackage{hyperref}
\usepackage{float}

\newcommand{\SOL}{\mathbf{SOL}} 
\newcommand{\N}{{\mathbb N}}

\newcommand{\Z}{{\mathbb Z}} 
\newcommand{\R}{{\mathbb R}} 
\newcommand{\C}{{\mathbb C}}

\newcommand{\cR}{\mathcal{R}} 
 
\newcommand{\bF}{\mathbf{F}} 

\newif \ifskip
\skiptrue
\newif \ifmargin
\marginfalse

\newtheorem{theorem}{Theorem}[section]
\newtheorem{corollary}[theorem]{Corollary}
\newtheorem{lemma}[theorem]{Lemma}
\newtheorem{claim}[theorem]{Claim}

\newtheorem{proposition}[theorem]{Proposition}

\newtheorem{definition}[theorem]{Definition}
\newtheorem{example}[theorem]{Example}
\newtheorem{remark}[theorem]{Remark}

\begin{document}

\title{On the Location of Roots of Graph Polynomials}

\author[J. Makowsky]{Johann A. Makowsky}
\address{Computer Science Department, Technion--IIT, Haifa, Israel}
\email{janos@.cs.technion.ac.il}
%\thanks{xxx}

\author[E. Ravve]{Elena V. Ravve}
\address{Software Engineering Department\\ORT-Braude College, Karmiel, Israel}
\email{cselena@braude.ac.il}
%\thanks{yyy}

\author[N. Blanchard]{Nicolas K. Blanchard}
\address{\'Ecole Normale Sup\'erieure, Paris, France}
\email{koliaza@gmail.com}
%\thanks{zzz}

\begin{abstract}
Roots of graph polynomials such as the characteristic polynomial, the chromatic polynomial,
the matching polynomial, and many others are widely studied. In this paper we examine
to what extent the location of these roots reflects the graph theoretic properties of
the underlying graph.
\end{abstract}

\maketitle
Version of \today
\tableofcontents
%---------------Introduction-------------------
%\newpage
\section{Introduction}
\label{se:intro}
\newif
\iferdos
%\erdostrue
\erdosfalse
%\marginfalse
%---------------------
A graph $G=(V(G), E(G))$ is given by the set of vertices $V(G)$ and a symmetric edge-relation $E(G)$.
We denote by $n(G)$ the number of vertices, and by $m(G)$ the number of edges.
$k(G)$ denotes the number of connected components of $G$.
We denote the class of finite graphs by $\mathcal{G}$.

Graph polynomials are graph invariants with values in a polynomial ring, usually $\Z[X_1, \ldots , X_r]$.
Let $P(G;X)$ be a graph polynomial.
A graph $G$ is $P$-unique if for all graphs $G'$ the identity of $P(G;X)$ and $P(G';X)$ implies that
$G$ is isomorphic to $G'$.
As a graph invariant $P(G;X)$ can be used to check whether two graphs are not isomorphic.
For $P$-unique graphs $G$ and $G'$ the polynomial $P(G;X)$ can also be used to check whether they are isomorphic.
One usually compares graph polynomials by their {\em distinctive power}.

With our definition of graph polynomials there are too many graph polynomials.
Traditionally, graph polynomials studied in the literature are definable in some logical formalisms.
However, in this paper
we only assume that our univariate graph polynomials are of the form
$$
P(G;X) = \sum_{i_1, \ldots, i_r=0}^{s(G)} h_{i_1, \ldots, i_r}(G) X_1^{i_1}\cdot \ldots \cdot X_r^{i_r}
$$
where $s(G)$ is a graph parameter with non-negative integers as its values,
and $h_{i_1, \ldots, i_r}(G): i_1, \ldots, i_r \leq s(G)$ are integer valued graph parameters.
All graph polynomials in the literature are of the above form\footnote{
In \cite{ar:MakowskyTARSKI,ar:AverbouchGodlinMakowsky10,ar:KotekMakowskyZilber11,ar:GodlinKatzMakowsky12} 
the class of graph polynomials definable in Second Order Logic
$\SOL$
is studied, which imposes that $s(G)$ and $h_i(G)=h(G;i)$ are definable in $\SOL$, 
which is stronger restriction. 
%However, for our discussion in this paper, definability in $\SOL$ is not needed.
}.
The logical formalism is not needed for the results in this paper, and introducing it here would only make the
paper less readable. Nevertheless, we shall indicate for the logically minded where the definability requirements
can be added without changing the results.

\subsection{Equivalence of graph polynomials}
Two graphs $G_1$ and $G_2$ are called {\em similar} if they have the same number of vertices, edges and connected components.
Two graph polynomials $P(G; X_1, \ldots X_r)$ and $Q(G;Y_1, \ldots, Y_s)$  are 
{\em equivalent in distinctive power (d.p-equivalent)} 
if for every two similar graphs $G_1$ and $G_2$
\begin{gather}
P(G_1,X_1, \ldots X_r) =P(G_2,X_1, \ldots X_r) \mbox{ iff }
Q(G_1,Y_1, \ldots Y_s) =Q(G_2,Y_1, \ldots Y_s).
\notag
\end{gather}

For a ring $\cR$
let $\cR^{\infty}$ denote the set of finite sequences of elements of $\cR$.
For a graph polynomial $P(G;X)$ we denote by $cP(G) \in \Z^{\infty}$ the sequence of coefficients of $P(G;X)$.
In Section \ref{se:equiv} we will prove the following theorem and some variations thereof:

\begin{proposition}
%\begin{prop}
\label{prop:equiv-0}
Two graph polynomials $P(G; X_1, \ldots X_r)$ and $Q(G;Y_1, \ldots, Y_s)$  are
d.p-equivalent) 
iff
there are two functions $F_1, F_2: \Z^{\infty} \rightarrow \Z^{\infty}$ such that for every graph $G$
\begin{gather}
F_1(n(G), m(G), k(G), cP(G)) = cQ(G) \mbox{ and } \notag \\
F_2(n(G), m(G), k(G), cQ(G)) = cP(G) \notag
\end{gather}
\end{proposition}
%\end{prop}
Proposition \ref{prop:equiv-0} shows that our definition of equivalence of graph polynomials
is mathematically equivalent to the definition proposed in
\cite{ar:MerinoNoble2009}.

%---------------------------------------------------------
\ifskip
\else
There are (possibly too) many equivalent graph polynomials. For example, let
$f: \N \rightarrow \N$ and
$g: \N \rightarrow \N$  be two injective functions and  let $P(G;X) = \sum_i a_i(G) X^i$ a graph polynomial. Then 
$Q(G;X) = \sum_i a_{f(i)(G)} X^{g(i)}$ is also a graph polynomial and $Q(G;X)$ is equivalent to $P(G;X)$.
\fi
%---------------------------------------------------------

\subsection{Reducibility using similarity}
In the literature one often wants to say that two graph polynomials are {\em almost the same}.
For example the various versions of the Tutte polynomial are said to be the same {\em up to a prefactor},
\cite{ar:Sokal2005a}, and the same holds for the various versions of the matching polynomial, \cite{bk:LovaszPlummer86}.
We propose a definition which makes this precise.
For this purpose we introduce the notion of {\em similarity functions}, defined in detail
in Section \ref{se:sim}, which captures the notion of {\em prefactor} as it is used in the literature.
A graph parameter is a {\em similarity function} if it is invariant under graph similarity.

%--------------------------------------------------------------------------
\ifskip
\else
They are inductively defined and are invariant under graph similarity.
Also in Section \ref{se:sim} we define precisely how to obtain univariate graph polynomials from
multivariate graph polynomials by
{\em substitutions of variables by similarity polynomials}.
\fi %skip
%--------------------------------------------------------------------------
Let 
$P(G; \bar{Y})$  and
$Q(G; \bar{X})$ 
be two multivariate graph polynomials with coefficients in a ring $\cR$.
We say that $P(G; \bar{X})$ is {\em prefactor reducible to} $Q(G; \bar{X})$ 
and we write 
$$
P(G; \bar{Y}) \preceq_{prefactor} Q(G; \bar{X})
$$ 
if there are similarity functions $f(G; \bar{X})$ and $g_1(G; \bar{X}), \ldots, g_r(G; \bar{X})$
such that
$$
P(G; \bar{Y}) = f(G; \bar{X}) \cdot Q(G; g_1(G; \bar{Y}), \ldots , g_r(G; \bar{Y}) ).
$$
$P(G; \bar{X})$  and
$Q(G; \bar{X})$ 
are {\em prefactor equivalent} if the relationship holds in both directions.
It follows that if
$P(G; \bar{X})$  and
$Q(G; \bar{X})$ 
are prefactor equivalent then they are d.p.-equivalent.

\subsection{Syntactic vs semantic properties of graph polynomials}
The notion of (semantic) equivalence of graph polynomials evolved very slowly, mostly in implicit arguments,
and is captured by our notion of d.p.-equivalence.
Originally, a graph polynomial such as the chromatic or characteristic polynomial
had a {\em unique} definition which both determined its algebraic presentation and its semantic content.
The need to spell out semantic equivalence emerged when the various forms of the Tutte polynomial had to be compared.
As was to be expected, some of the presentations of the Tutte polynomial had more convenient properties than other,
and some of the properties of one form got completely lost when passing to another semantically equivalent form.
Let us make this clearer via examples:
\begin{enumerate}
\item
The property that a graph polynomial $P(G;X)$ is 
{\em monic}\footnote{A univariate polynomial is monic if the leading coefficient equals $1$.} 
for each graph $G$ has no semantic content, because
multiplying each coefficient by a fixed integer gives an equivalent graph polynomial.
\item
Similarly, proving that the leading coefficient of $P(G;X)$ equals the number of vertices of $G$ 
is semantically meaningless, for the same reason.
However, proving that two graphs $G_1, G_2$ with $P(G_1,X) = P(G_2,X)$ have the same number of vertices is semantically meaningful.
\item
In similar vain, the classical result that the characteristic polynomial of a tree equals the 
(acyclic) matching polynomial of the same tree,
is a syntactic coincidence, or reflects a clever choice in the definition of the acyclic matching polynomial, but it is
semantically speaking meaningless. 
The semantic content of this theorem says that if we restrict our graphs to trees, then the characteristic and the
matching polynomials (in all its versions) have the same distinctive power on trees of the same size.
\end{enumerate}

\subsection{Roots of graph polynomials}
The literature on graph polynomials mostly got its inspiration from the successes in studying the  chromatic polynomial
and its many generalizations and the
characteristic polynomial of graphs. 
In both cases the roots of graph polynomials are given much attention and are meaningful when these polynomials model physical reality.

A complex number $z \in \C$ is a root  of a univariate graph polynomial $P(G;X)$
if there is a graph $G$ such that $P(G;z)=0$.
It is customary to study the
location of the roots of univariate  graph polynomials. 
Prominent examples, besides
the chromatic polynomial, the matching polynomial and the
characteristic polynomial, are the independence polynomial, the domination polynomial and the vertex cover polynomial.

For a fixed graph polynomial $P(G;X)$ typical statements about roots are:
\begin{enumerate}
\item
For every $G$ the roots of $P(G;X)$ are real.
This is true for the characteristic and the matching polynomial \cite{bk:CvetkovicDoobSachs1995,bk:LovaszPlummer86}.
\iferdos
\item
Assuming that all roots of $P(G;X)$ are real,
the (second) largest root has an  interesting combinatorial interpretation.
This is true for the characteristic polynomial 
\cite[Chapter 4]{bk:BrouwerHaemers2012}.
\item
The multiplicity of a certain value $a$ as a root of $P(G:X)$ has an interesting interpretation.
For example the multiplicity of $0$ as a root of the Laplacian polynomial is the number
of connected components of $G$, see
\cite[Chapter 1.3.7]{bk:BrouwerHaemers2012}.
\else \fi %erdos
\item
For every $G$ all real roots of $P(G;X)$ are positive (negative)
or the only real root is $0$.
The real roots are positive in the case of the chromatic polynomial and the clique polynomial,
and negative for the independence polynomial 
\cite{bk:DongKohTeo2005,ar:HoedeLi94,ar:BrownHickmanNowakowski2004,ar:GoldwurmSantini2000,phd:Hoshino}.
\item
For every $G$ the roots of $P(G;X)$ are contained in a disk of radius
$\rho(d(G))$ where $d(G)$ is the maximal degree of the vertices of $G$.
This is the case for the chromatic polynomial
\cite{bk:DongKohTeo2005,ar:Sokal01}.
\item
For every $G$ the roots of $P(G;X)$ are contained in a disk of constant radius.
This is the case for the edge-cover polynomial \cite{ar:CsikvariOboudi2011} 
\item
The roots of $P(G;X)$ are dense in the complex plane.
This is again true for the chromatic polynomial, the dominating polynomial and the independence
polynomial \cite{bk:DongKohTeo2005,ar:Sokal04,ar:BrownHickmanNowakowski2004,phd:Hoshino}.
\end{enumerate}
In  Section \ref{se:catalog}
%an appendix 
we give a 
more detailed discussion
%catalogue 
of graph polynomials for which the location of their roots was studied in
the literature.

\subsection{Main results}
In this paper we address the question on how the particular location of the roots of
a univariate graph polynomial behaves under d.p-equivalence and prefactor equivalence.
Our main results, proved in Section \ref{se:main} are the following {\em modification theorems},
so called, because they show how to modify the location of the roots of
a graph polynomial within its equivalence class.

\begin{itemize}
\item
{\bf Theorems \ref{th:main1} and \ref{th:main1a}}:
For every univariate graph polynomial 
$P(G;X) = \sum_{i=0}^{s(G)} h_i(G) X^i$
where  $s(G)$ and $h_i(G), i=0, \ldots s(G)$
are graph parameters with values in $\N$, 
there exists a 
univariate graph polynomials $Q_1(G;X)$, 
prefactor equivalent to $P(G;X)$ such that
for every $G$ all real roots of $Q_1(G;X)$ are positive (negative)
or the only real root is $0$.
\item
{\bf Theorems \ref{th:main2}}:
Let $s(G)$ be a similarity function.
For every univariate graph polynomial with integer (real) coefficients 
$P(G;X) = \sum_{i=0}^{i= s(G)} h_i(G) X_i$
\\
there is a d.p.-equivalent graph polynomial 
$Q(G;X) = \sum_{i=0}^{i= s(G)} H_i(G) X_i$
with integer (real) coefficients such that all the roots of $Q(G;X)$ are real.
\item
{\bf Theorems \ref{th:main3} and \ref{th:main3a}}:
For every univariate graph polynomial $P(G;X)$ there exist 
univariate graph polynomials $Q_2(G;X)$ prefactor equivalent to
$P(G;X)$ such that
$Q_2(G;X)$ has only countably many roots, and the roots are dense in the complex plane.
If we want to have
all roots real and dense in $\R$, we have to require d.p.-equivalence.
\item
{\bf Theorem \ref{th:main4} and Corollary \ref{cor:realbounded}}:
For every univariate graph polynomial $P(G;X)$ there exist 
univariate graph polynomials $Q_3(G;X)$ prefactor equivalent to
$P(G;X)$ such that
for every $G$ the roots of $Q_3(G;X)$ are contained in a disk of constant radius.
If we want to have all roots real and bounded in $\R$, we have to require d.p.-equivalence.
\end{itemize}
We will discuss in Section \ref{se:conclu}
what kind of restrictions one might impose on graph polynomials such as to make the
location of the roots more meaningful.

\subsubsection*{Acknowledgments}
We would like to thank
I. Averbouch,
P. Csikv{\'a}ri,
J. Ellis-Monaghan,
P. Komj{\'a}th, 
T. Kotek,
 N. Labai and
A. Shpilka
for encouragement and valuable discussions.
A preliminary extended abstract and poster was presented at the Paul Erd\"os Centennial Conference
in Budapest \cite{pr:MakowskyRavveErdos}.

%---------------Equivalence-------------------
\section{Equivalence of graph polynomials}
\label{se:equiv}
The results of this sections were first used in the lecture notes \cite{up:Lecture-11}
by the first two authors in 2009.
% http://www.cs.technion.ac.il/~janos/COURSES/236605-09/lec11.pdf

Recall that two graphs $G_1, G_2$ are {\em similar} if
$n(G_1) =n(G_2),
m(G_1) =m(G_2)$ and $k(G_1) =k(G_2)$.

\subsection{Distinctive power on similar graphs}

%We now compare graph polynomials in essentially two ways, with respect to similar graphs
%and with respect to all graphs.
\begin{definition}
Let $P$ and $Q$ be two graph polynomials.
\begin{enumerate}
\item
{\em $P$ is more distinctive as $Q$, 
$Q \preceq_{d.p} P$
} if for all pairs of similar graphs
$G_1, G_2$ with $Q(G_1) = Q(G_2)$ we also have $P(G_1)=P(G_2)$.
%\item
%{\em $P$ is at strongly more distinctive as $Q$, $Q \preceq_{s.d.p} P$} if for all pairs of
%(not necessarily  similar) graphs
%$G_1, G_2$ with $Q(G_1) = Q(G_2)$ we also have $P(G_1)=P(G_2)$.
\item
$P$ and $Q$ are {\em d.p.-equivalent or equally distinctive, $P \sim_{d.p} Q$}, 
if both
$Q \preceq_{d.p} P$ and $P \preceq_{d.p} Q$ hold.
%\item
%$P$ and $Q$ are {\em strongly d.p.-equivalent or strongly equally distinctive, $P \sim_{s.d.p} Q$}, if both
%$Q \preceq_{s.d.p} P$ and $P \preceq_{s.d.p} Q$ hold.
\end{enumerate}
\end{definition}

\subsection{Examples of d.p.-equivalent graph polynomials}
\label{se:equiv-ex}
\begin{example}
\label{ex:matching}
Let $m_k(G)$ denote the number independent sets of edges of size $k$.
There are two versions of the univariate matching polynomial, cf. \cite{bk:LovaszPlummer86}:
The {\em matching defect polynomial} (or {\em acyclic polynomial})
$$
\mu(G,\lambda)= \sum_k^{\frac{n}{2}} (-1)^k m_k(G) \lambda^{n-2k},
$$
and the
\em{matching generating polynomial}
$$
g(G,\lambda)= \sum_k^n m_k(G) \lambda^k
$$
The relationship between two is given by
$$
\mu(G,\lambda)=\sum_k^\frac{n}{2}(-1)^k m_k(G) \lambda^{n-2k} =
\lambda^{n}\sum_k^\frac{n}{2}(-1)^k m_k(G) \lambda^{-2k} =
$$
and
$$
=\lambda^{n}\sum_k^\frac{n}{2} m_k(G) ((-1)\cdot \lambda^{-2})^{k}
=\lambda^{n}\sum_k^\frac{n}{2} m_k(G) (-\lambda^{-2})^{k} =
\lambda^ng(G,(-\lambda^{-2}))
$$
It follows that $g$ and $m$ are equally distinctive with respect to similar graphs.
However,
$g(G;X)$ is invariant under addition or removal of isolated vertices,
whereas $\mu(G;X)$ counts them.
\end{example}

\begin{example}
\label{ex:coefficients}
Let $P(G,X)$ be a univariate graph polynomial with integer coefficients and
$$
P(G,X) = \sum_{i=0}^{d(G)} a_i(G) X^i = \sum_{i=0}^{d(G)} b_i(G) X_{(i)}
=
\sum_{i=0}^{d(G)} c_i(G) {X \choose i}
=
\prod_i^{d(G)} (X - z_i)
$$
where $X_{(i)} = X(X_1)\cdot \ldots \cdot (X-i+1)$ is the falling factorial function.
We denote by  $aP(G)=(a_0(G), a_1(G), \ldots, a_{d(G)}(G))$,
$bP(G) =(b_0(G), b_1(G), \ldots, b_{d(G)}(G))$ 
and
$cP(G) =(c_0(G), c_1(G), \ldots, c_{d(G)}(G))$ 
the coefficients of these polynomial presentations
and by $zP(G) = (z_1, \ldots , z_{d(G)}$ the roots of these polynomials with their multiplicities.
We note that the four presentations of $P(G;X)$
are all d.p.-equivalent.
\end{example}

\begin{example}
Let $G=(V(G), E(G))$ be a loop-less graph without multiple edges.
Let $A_G$ be the adjacency matrix of $G$, $D_G$ the diagonal matrix with $(D_G)_{i,i} =d(i)$, the degree of the vertex $i$,
and $L_G = D_G -A_G$.
In spectral graph theory two graph polynomials are considered, the {\em characteristic polynomial of $G$},
here denoted by $P_A(G;X) = \det (X\cdot \mathbf{I}- A_G)$, and the {\em Laplacian polynomial}, 
here denoted by $P_L(G;X) = \det (X\cdot \mathbf{I}- L_G)$. Here $\mathbf{I}$ denotes the unit element in
the corresponding matrix ring.
%
%\begin{center}
%\begin{verbatim}
%
%                               |
%_/|\_                        |\|/|
% \|/                         |/ \|
%
%  G                            H
%
%\end{verbatim}
%\end{center}
%
$G$ and $H$ in Figure \ref{fig1}  are similar.
We have 
$$P_A(G;X)=P_A(H;X)= (X-1)(X+1)^2(X^3 -X^2 -5X +1),$$ but $G$ has two spanning trees, and $H$ has six.
Therefore, $P_L(G;X) \neq P_L(H;X)$, as one can compute the number of spanning trees from $P_L(G;X)$.
For more details, cf. \cite[Exercise 1.9]{bk:BrouwerHaemers2012}.

\begin{figure}[h]
\begin{center}
\includegraphics[width=.15\linewidth]{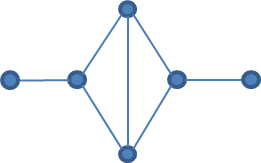} 
\hspace{2cm}
\includegraphics[width=.07\linewidth]{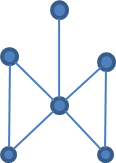} 
\\
$G$ \hspace{4cm} $H$
\end{center}
\caption{}
\label{fig1}
\end{figure}

On the other hand, $G'$ and $H'$ in Figure \ref{fig2} are similar, but
$G'$ is not bipartite, whereas, $H'$ is.
\begin{figure}[h]
\begin{center}
\includegraphics[width=.15\linewidth]{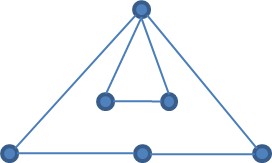} 
\hspace{2cm}
\includegraphics[width=.15\linewidth]{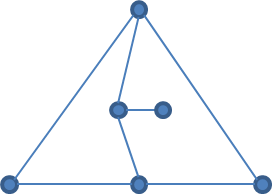} 
\\
$G'$ \hspace{4cm} $H'$
\caption{}
\label{fig2}
\end{center}
\end{figure}
%
%\begin{center}
%\begin{verbatim}
%
%   |\                      |\
%  /|\\                    /| \
% / |_\\                  / |_ \
%/__o___\                /__|___\
%
%   G'                       H'
%       
%\end{verbatim}
%\end{center}
%
Hence $P_A(H;X) \neq P_L(G',X)$, but
$P_L(H;X)=P_L(G';X)$.
See, \cite[Lemma 14.4.3]{bk:BrouwerHaemers2012}.

{\bf Conclusion:} The characteristic polynomial and the Laplacian polynomial
are d.p.-incomparable.
However, if restricted to $k$-regular graphs, they are d.p.-eqivalent, cf. \cite{bk:BrouwerHaemers2012}.
\end{example}

\subsection{Characterizing d.p.-equivalence}
\begin{proposition}
\label{prop:equiv}
Let $P$ and $Q$ be two graph polynomials with coefficients 
in a ring $\cR$ which contains the natural numbers $\N$, and denote by
$cP$ and $cQ$ respectively the sequence of their coefficients.

The following are equivalent:
\begin{enumerate}
\item %(i)
$Q \preceq_{d.p} P$; 
\item %(ii)
there is a function $F_1: \cR^{\infty} \rightarrow \cR^{\infty}$ such for every graph $G$
\begin{gather}
F_1(n(G), m(G), k(G), cP(G)) = cQ(G).  \notag
%F_1(aP(G)) = aQ(G)  \notag
\end{gather}
\end{enumerate}
%--------------------------------------------------------------
\ifskip
\else
\item %(iii)
there is a function $F_2: \Z^{\infty} \rightarrow \Z^{\infty}$ such for every graph $G$
\begin{gather}
%F_2(n(G), m(G), k(G), bP(G)) = bQ(G)  \notag
F_2(bP(G)) = bQ(G)  \notag
\end{gather}
\item %(iv)
there is a function $F_3: \Z^{\infty} \rightarrow \Z^{\infty}$ such for every graph $G$
\begin{gather}
%F_3(n(G), m(G), k(G), zP(G)) = zQ(G)  \notag
F_3(zP(G)) = zQ(G)  \notag
\end{gather}
\end{enumerate}
The corresponding also holds for
$Q \preceq_{d.p} P$ 
and
$$F(n(G), m(G), k(G), aP(G)) = aQ(G).$$
\fi %skip
%-------------------------------------------------------------------
\end{proposition}
\begin{proof}
We prove the equivalence of (i) and (ii).
The equivalences follow from the fact that the coefficients and the  roots with their multiplicities
determine a univariate polynomial uniquely.
The proof for
$Q \preceq_{d.p} P$ is analogous.

(i) $\rightarrow$ (ii):\\
Let $S$ be a set of finite graphs and
$s \in \Z^{\infty}$. 
For a graph polynomial $P$  we define:
\begin{gather}
P[S] = \{ s \in \Z^{\infty}: aP(G)=s \mbox{ for some } G \in S\} \notag\\
P^{-1}(s)= \{G : aP(G)=s\}. \notag
\end{gather}
Now assume $P(G,X) \preceq_{d.p.} Q(G,X)$. 
\\
If $Q^{-1}(s) \neq \emptyset$, then for every $G_1, G_2 \in Q^{-1}(s)$
we have 
$cQ(G_1) = cQ(G_2)$, and therefore
$cP(G_1) = cP(G_2)$. 
Hence $P[Q^{-1}(s)] = \{t_s\}$ for some $t_s \in \Z^{\infty}$.
Now we define 
\begin{gather}
F_{P,Q}(s)=
\begin{cases}
t_s & Q^{-1}(s) \neq \emptyset\\
s & \mbox{else}
\end{cases}
\notag
\end{gather}
%---------------------------------------------

(ii) $\rightarrow$ (i): \\
Assume there is a function
$F: \Z^{\infty} \rightarrow \Z^{\infty}$ such that 
for all graphs $G$ we have 
$ F(aQ(G)) = aP(G)$.

Now let $G_1, G_2$ be similar graphs such that 
$Q(G_1) = Q(G_2)$.

Clearly we have $aQ(G_1) = aQ(G_2)$.
Hence
$F(aQ(G_1)) = F(aQ(G_2))$.

Since for all $G$ we have $ F(aQ(G)) = aP(G)$,
we get $aP(G_1)= aP(G_2)$ and therefore 
$P(G_1)= P(G_2)$. 
\end{proof}
\begin{remark}
\begin{enumerate}
\item
As we have seen in Example \ref{ex:coefficients},
instead of the coefficients $cP$ and $cQ$ one could consider any other sequence
of elements which characterize the coefficients, or in the univariate case, also the
sequence of the roots of the polynomials.
\item
The theorem also holds in a restricted version, where all the graphs considered
have a certain graph property $\mathcal{P}$.
\end{enumerate}
\end{remark}

Using Proposition \ref{prop:equiv} it is now easy to construct many strongly d.p.-equivalent polynomials:

\begin{corollary}
\label{cor:many}
Let $z: \C \rightarrow \C$ be an injective complex function.
Let $G$ be a graph
and let $P(G;X)$ be a univariate graph polynomial with roots
$\theta_i(G): i \leq d(G)$, i.e., 
$P(G;X)= \prod_{i \leq d(G)} (X- \theta_i(G))$.
Let $P_z(G;X)= \prod_{i \leq d(G)} (X- z(\theta_i(G)))$.
Then $P(G;X)$ and $P_z(G;X)$ are d.p.-equivalent.
\end{corollary}

As already mentioned in the introduction,
it is therefore reasonable to restrict the possibilities of creating graph polynomials
by imposing some restricting conditions on the representability of the graph polynomials.
But one has to be careful not to be too restrictive.
A good candidate for such a restriction is
the class of graph polynomials definable in Second Order Logic $\SOL$
studied in \cite{ar:MakowskyTARSKI,ar:AverbouchGodlinMakowsky10,ar:KotekMakowskyZilber11,ar:GodlinKatzMakowsky12}. 
However, for our discussion in this paper the precise definition of definability in $\SOL$ is not needed.

%---------------Similarity-------------------
\section{Similarity function and  prefactor reductions}
\label{se:sim}
\subsection{Prefactor equivalence}
A graph parameter $f(G)$ with values in some function space $\bF$ over some ring $\cR$ 
is called a 
{\em similarity function} if for any two similar graphs $G,H$ we have that $f(G)=f(H)$.
If $\bF$ is  a subset of the set of analytic functions we speak of
{\em analytic similarity functions}.

If $\bF$ is the polynomial ring $\Z[\bar{X}]$ with
set of indeterminates $\bar{X}= (X_1, \ldots X_r)$,
we speak of {\em  similarity polynomials}.
It will be sometimes useful to allow classes of functions spaces 
which are closed under
{\em reciprocals} and {\em inverses} rather
than just similarity polynomials.  

\begin{example}
Typical examples of similarity functions
are 
\begin{enumerate}
\item
The nullity $\nu(G) = m(G)- n(G) +k(G)$ and the rank $\rho(G) =n(G)- k(G)$ of a graph $G$
are similarity polynomials with integer coefficients.
\item
Similarity polynomials can be formed inductively starting with
similarity functions $f(G)$
not involving indeterminates, and
monomials of the form $X^{g(G)}$ where $X$ is an indeterminate and 
$g(G)$ is a 
similarity function
not involving indeterminates.
One then closes under pointwise addition, subtraction, multiplication and substitution of
indeterminates $X$ by similarity polynomials.
\item
$f(G; X) = n(G)X^2$ is a similarity polynomial with integer coefficients.
Its inverse $f^{-1}(G; X) = n(G)^{-1} X^{\frac{1}{2}}$ is analytic at any point $a\in \R$ with $a \neq 0$.
Its reciprocal $\frac{1}{f(G;X)}$ is rational.
\end{enumerate}
\end{example}

In the literature one often wants to say that two graph polynomials are {\em almost the same}.
We propose a definition which makes this precise.

\begin{definition}
Let 
$P(G; Y_1, \ldots, Y_r)$  and
$Q(G; X_1, \ldots X_s)$ 
be two multivariate graph polynomials with coefficients in a ring $\cR$.
\begin{enumerate}
\item
We say that $P(G; \bar{X})$ is {\em prefactor reducible to $Q(G; \bar{X})$ 
over a set of similarity functions $\bF$},
and we write 
$$
P(G; \bar{Y}) \preceq_{prefactor}^{\bF} Q(G; \bar{X})
$$ 
if there are similarity functions $f(G; \bar{X})$ and $g_i(G; \bar{X}), i \leq r$ in $\bF$
such that
$$
P(G; \bar{Y}) = f(G; \bar{X}) \cdot Q(G; g_1(G; \bar{Y}), \ldots, g_r(G;\bar{Y}))
$$
\item
We say that $P(G; \bar{X})$ is {\em substitution reducible to} $Q(G; \bar{X})$ over $\bF$
and we write 
$$
P(G; \bar{Y}) \preceq_{subst} Q(G; \bar{X})
$$ 
if $f(G;\bar{X})=1$ for all values of $\bar{X}$.
\item
We say that
$P(G; \bar{X})$  and
$Q(G; \bar{X})$  are {\em  prefactor equivalent},
and we write
$$
P(G; \bar{X}) \sim_{prefactor} Q(G; \bar{X})
$$ 
if the relation holds in  both directions.
\item
substitution equivalence
$ P(G; \bar{X}) \sim_{subst} Q(G; \bar{X}) $ 
is defined analogously.
\end{enumerate}
\end{definition}

The following properties follow from the definitions.
\begin{proposition}
Assume we have two graph polynomials
$P(G; \bar{X})$  and
$Q(G; \bar{X})$. 
For reducibilities we have:
\begin{enumerate}
\item
$ P(G; \bar{X}) \preceq_{subst} Q(G; \bar{X}) $ implies
$ P(G; \bar{X}) \preceq_{prefactor} Q(G; \bar{X}) $.
\item
$ P(G; \bar{X}) \preceq_{prefactor} Q(G; \bar{X}) $ implies
$ P(G; \bar{X}) \preceq_{d.p.} Q(G; \bar{X}) $.
\end{enumerate}
The corresponding implications for equivalence obviously also hold.
\end{proposition}

\subsection{The classical examples}
\begin{example}[The universal Tutte polynomial]
Let $T(G;X,Y)$ be the Tutte polynomial, cf. \cite[Chapter 10]{bk:Bollobas98}.
The universal Tutte polynomial is defined by
$$
U(G;X,Y,U,V,W) = U^{k(a)G} \cdot V^{\nu(G)} \cdot W^{\rho(G)} \cdot T\left(G; \frac{UX}{W}, \frac{Y}{U}\right)
$$
$U(G;X,Y,U,V,W)$ is the most general graph polynomial satisfying the recurrence relations
of the Tutte polynomial in the sense that every other graph polynomial satisfying these recurrence
relations is a substitution instance of $U(G;X,Y,U,V,W)$.

Clearly, $U(G;X,Y,U,V,W)$ is 
prefactor equivalent to $T(G;X,Y)$ using rational similarity functions.
\end{example}

\begin{example}[The matching polynomials]
In Example \ref{ex:matching} we have already seen
%Let $m_i(G)$ be the number of $i$-matchings of the graph $G$.
%There are 
the three matching polynomials:
\begin{gather}
\mu(G;X) = \sum_i (-1)^i m_i(G) X^{n(G)- 2i}
\notag \\
g(G;Y) = \sum_i m_i(G) Y^i
\notag \\
M(G; X,Y) = \sum_i m_i(G) X^i Y^{n(G)- 2i}
\notag
\end{gather}
We have
$ \mu(G;X) = X^{n(G)} \cdot g(G; -X^{-2}) $
and 
$M(G;X,Y) = Y^{n(G)} \cdot g(G; \frac{X}{Y^2})$.
Clearly, all three matching polynomials are mutually 
prefactor bi-reducible using analytic similarity functions.
\end{example}

\begin{example}
The following graph polynomials are d.p.-equivalent but incomparable by prefactor reducibility:
\begin{enumerate}
\item
$M(G;X)$ and $M(G;X)^2$;
\item
$\mu(G;X)$ and $\sum_i m_i(G){{X}\choose{i}}$.
\end{enumerate}
\end{example}
%---------------Locations-------------------
\section{Location of the roots of equivalent graph polynomials}
\label{se:main}
In this section we study the location of the roots of a graph polynomial.
In particular we are interested in the question of whether
all its roots are real, whether they are dense in $\R$ or in $\C$, or whether their absolute value is
bounded independently of the graph.
We first discuss these properties on known graph polynomials,
and then we show that up to d.p-equivalence (or even prefactor or substitution equivalence)
these properties can be forced to be true.
%-----------------------------------------------------------------------------------------
\subsection{Known graph polynomials and their roots}
\label{se:catalog}
\subsubsection{Characteristic polynomials of symmetric matrices}

It is a classical result of linear algebra that the characteristic polynomial
$det(X \cdot \mathbf{I} -A)$ where $A$ is a symmetric real matrix
has only real roots.

Let $G (V(G),E(G))$ be a simple graph. For $v \in V(G)$ let $d(v)$ be the degree of $v$, 
and for $e=(u,v) \in E(G)$ let $c(u,v)$ be length of the shortest proper cycle
containing $e$ if such a cycle exists, otherwise we set $c(u,v) =1$ if $(u,v) \in E(G)$
and $0$ otherwise.

We can associate various symmetric matrices with graphs:
the adjacency matrix $A_G$, where the diagonal elements are all $0$, the Laplacian $L_G$,
where the diagonal elements are give by $a_{v,v} =d(v)$,  are both 
used in the literature, cf. \cite{bk:BrouwerHaemers2012}.
The graph polynomials 
$$
P_A(G;X)= \det( X\cdot \mathbf{I} - A_G)
$$
and
$$
P_L(G;X)= \det( X\cdot \mathbf{I} - L_G)
$$
are the characteristic polynomial, respectively the Laplacian polynomial from Section \ref{se:equiv-ex}. 
All their roots are real,
and $P_A(G;X)$ and $P_L(G;X)$ are not d.p.-equivalent.

We can let our imagination also run wild.
Define for example  $C_G$ with
$$
(C_G)_{u,v} = \begin{cases}
c(u,v) & (u,v) \in E(G) \\
d(v) & u=v
\end{cases}
$$
and  define the {\em characteristic cycle polynomial $P_{cc}(G, X)$} by
$$
P_{cc}(G, X) = det(X \cdot \mathbf{I} - C_G)
$$
We know by construction that all the roots of $P_{cc}(G, X)$ are real, but we do not know whether
this is an interesting graph polynomial.
The reader can now try to construct infinitely many pairwise non-d.p.-equivalent graph polynomials 
where all the roots are real.

\subsubsection{The matching polynomials}

We have already defined the three matching polynomials,
$\mu(G;X), g(G;X)$ and $M(G;X,Y)$ in Section \ref{se:equiv-ex}.
For more background, cf. \cite{bk:LovaszPlummer86}.
The roots of $\mu(G;X)$ have interpretations in chemistry, cf. \cite{bk:Trinajstic1992,ar:Balaban93,ar:Balaban95}.

\begin{theorem}[\cite{ar:HeilmannLieb72}]
The roots of $\mu(G;X)$ and $g(G;X)$ are all real.
The roots of $\mu(G;X)$ are symmetrically placed around $0$ and the roots of $g(G;X)$
are all negative.
\end{theorem}
\begin{proof}[Sketch of proof]
One first proves it for $\mu(G;X)$ and derives the statements for $g(G;X)$.
First one notes that on forests $F$ the characteristic polynomial of $F$ satisfies
$$
P_A(F;X) = \mu(F;X).
$$
Therefore the roots of $\mu(F;X)$ are real for forests.
Then one shows that for each graph $G$ there is a forest $F_G$
such that $\mu(G;X)$ divides $\mu(F_G;X)$.
\end{proof}

\subsubsection{The chromatic polynomial}
We first define a parametrized graph parameter $\chi(G;k)$ for natural numbers $k$
as the number of proper $k$-colorings of the graph $G$.
Birkhoff in 1912 showed that this is a polynomial in $k$ and therefore can be extended
as a polynomial $\chi(G;X)$ for complex values for $X$.
The most complete reference on the chromatic polynomial is
\cite{bk:DongKohTeo2005}.
The roots of the chromatic polynomial have interpretations in statistical mechanics.
For our discussion in this paper the following theorem summarizes what we need:

\begin{theorem}[A. Sokal]
\ 
\begin{description}
\item[\cite{ar:Sokal04}]
The roots of $\chi(G;X)$ are dense in the complex plane.
\item[\cite{ar:Sokal01}]
The absolute values of the roots of $\chi(G;X)$ are bounded by a function of the degree of $G$.
\end{description}
\end{theorem}
There are several variations of the chromatic polynomial, 
\cite{ar:BrentiRoyleWagner1994,bk:DongKohTeo2005}:
the $\sigma$-polynomial and the $\tau$-polynomial which are both d.p.-equivalent to $\chi(G;X)$,
and the adjoint polynomial, which is d.p.-equivalent to $\chi(\bar{G};X)$, the chromatic polynomial
of the complement graph. Clearly, also the roots of the adjoint polynomial are dense in the complex plane,
because every graph is the complement of some other graph.

\subsubsection{The independence polynomial}
Let $in_i(G)$ denote the number of independent sets of size $i$ of $G$.
The independence polynomial is defined as
$$
In(G;X) = \sum_{i=0}^{i=n(G)} in_i X^i
$$
and was introduced first in \cite{ar:GutmanHarary83}.
A comprehensive survey may be found in \cite{pr:LevitMandrescu05}.
For our discussion in this paper the following theorem summarizes what we need:

\begin{theorem}[\cite{ar:BrownHickmanNowakowski2004}]
\ 
\begin{enumerate}
\item
The complex roots of $In(G;X)$ are dense in the complex plane.
\item
The real roots of $In(G;X)$ are all negative and are dense in $(-\infty, 0]$.
\end{enumerate}
\end{theorem}

Two graph polynomials are related to the independence polynomial:
The {\em clique polynomial} defined by
$$
Cl(G;X) = \sum_{i=0}^{i=n(G)} cl_i X^i = In(\bar{G};X)
$$
where $cl_i(G)$ denotes the numbers of cliques of size $i$ of $G$ and $\bar{G}$ is the complement graph of $G$,
and the {\em vertex cover polynomial}, defined by
$$
Vc(G;X) = \sum_{i=0}^{i=n(G)} vc_i X^i = X^{-n} In(G;X^{-1})
$$
where $vc_i(G)$ denotes the numbers of vertex covers of size $i$ of $G$.
Note that $A$ is vertex cover of $G$ if and only if $V(G)-A$ is an independent set of $G$.

\begin{proposition}
\begin{enumerate}
\item
$Vc(G;X)$ and $In(G;X)$ are prefactor equivalent.
\item
$Cl(G;X)$ and $In(G;X)$ are not d.p.-equivalent.
\end{enumerate}
\end{proposition}
\begin{proof}
(i) We use $Vc(G;X) = X^{-n} In(G;X^{-1})$.
\\
(ii) 
Let $C_n$ be the graph on $n$ vertices which is connected and regular of degree $2$,
and $K_n$ be the complete graph on $n$ vertices. We denote by $G + H$ the disjoint union of
the graphs $G$ and $H$.
\\
We look at the graph $C_4$ and $C_3 +K_1$ and compute
$In(C_4, 1) =6$ but $In(C_3 + K_1, 1) =7$, whereas
$Cl(C_4, 1) = In(C_3 + K_1, 1) =8$. 
\end{proof}
For a discussion of clique polynomials, cf. \cite{ar:HoedeLi94}.
$Vc(G;X)$ was first introduced in \cite{ar:DongHendyTeoLittle02}.
For a detailed discussion of these polynomials, cf. \cite{phd:Averbouch}.

\begin{corollary}
The roots of $Cl(G;X)$ and $Vc(G;X)$ are dense in $\C$.
\end{corollary}
\begin{proof}
For $Cl(G;X)$ we use that every graph is the complement of some graph.
\\
For $Vc(G;X)$ we use that if a set $S \subseteq \C$ is dense, the so is the set $\{ z^{-1}: z \in \C\}$.
\end{proof}

\subsubsection{The domination polynomial}
Let $d_i(G)$ denote the number of dominating sets of size $i$ of $G$.
The domination polynomial is defined as
$$
D(G;X) = \sum_{i=0}^{i=n(G)} d_i X^i
$$
and was introduced first in \cite{ar:ArochaLlano2000} and further studied in
\cite{phd:Alikhani,ar:AkbariAlikhaniPeng2010,ar:KotekPreenSimonTittmanTrinks2012}.
\\
For our discussion in this paper the following theorem summarizes what we need:
\begin{theorem}[\cite{ar:BrownTufts2013}]
The complex roots of $D(G;X)$ are dense in the complex plane.
\end{theorem}

\subsubsection{The edge cover polynomial}
Let $e_i(G)$ denote the number of edge covers of size $i$ of $G$.
The edge cover polynomial is defined as
$$
E(G;X) = \sum_{i=0}^{i=n(G)} e_i X^i
$$
and was introduced first in \cite{ar:AkbariOboudi2013,ar:CsikvariOboudi2011}.
The roots of $E(G;X)$ are bounded. More precisely:

\begin{theorem}[\cite{ar:CsikvariOboudi2011}]
All roots of $E(G,X)$ are in the ball 
$$\{ z \in \C: |z| \leq \frac{(2+ \sqrt{3})^2}{1+\sqrt{3}} = \frac{(1+\sqrt{3})^3}{4} \}.$$
\end{theorem}

%------------------------------------------------------------------------
\subsection{Real roots}
We first study the location of the roots of graph polynomials which are generating functions.
Let $s(G)$ and
$h_i(G), i \leq s(G)$ be graph parameters
which take values in $\N$, and such that $h_i(G)=0$ for $i > s(G)$.

Let $P(G;X)$ be defined by 
\begin{gather}
\label{P}
P(G;X)= \sum_{i=0}^{i=s(G)} h_i(G) X^i.
\end{gather}
Clearly, $P(G;X)$ has no strictly positive real roots.
We want to find $P'(G;X)$ which is prefactor reducible to $P(G;X)$ 
with no negative real roots.

First we formulate two lemmas.
\begin{lemma}
\label{le:nonegative}
Let 
\begin{gather}
P(X) = \sum_{i=0}^{d} (-1)^i h_i X^i
\end{gather}
be a polynomial, where all the $h_i \in \N$ and at least for one $i \leq d$
the coefficient $h_i \neq 0$.
Then
\begin{gather}
P(X) = \sum_{i=0}^{d} (-X)^i h_i  > 0
\end{gather}
whenever $X$ is assigned a negative real.
\end{lemma}

\begin{lemma}
\label{le:norealroots}
Let 
\begin{gather}
P(X) = \sum_{i=0}^{d} h_i X^{2i}
\end{gather}
be a polynomial, where all the $h_i \in \N$ and at least for one $i \leq d$
the coefficient $h_i \neq 0$.
Then
\begin{gather}
P(X) = \sum_{i=0}^{d} X^{2i} h_i  > 0
\end{gather}
whenever $X$ is assigned a real.
\end{lemma}

\begin{theorem}
\label{th:main1}
Let $P(G;X)$ be as above.
Then there exist two univariate graph polynomials
$P_1(G;X)$ and $P_2(G;X)$ which are substitution equivalent
to $P(G;X)$ such that
\begin{enumerate}
\item
$P_1(G;X)$ has no negative real roots, and
\item
$P_2(G;X)$ has no real roots besides possibly $0$.
\end{enumerate}
\end{theorem}

\begin{proof}
We put
$$
P_1(G;X) = P(G; -X) = \sum_{i=0}^{i=s(G)} (-1)^i h_i(G) X^i
$$
and
$$
P_2(G;X) = P(G; X^2) = \sum_{i=0}^{i=s(G)}  h_i(G) X^{2i}.
$$
Clearly, both polynomials are substitution equivalent to $P(G;X)$ using analytic functions independent of the graph $G$.
Using Lemma \ref{le:nonegative} we see that $P_1(G;X)$ has negative real roots.
Using Lemma \ref{le:norealroots} we see that $P_2(G;X)$ has no real roots except for possibly $0$.
\end{proof}

To treat the general case where $h_i(G)$ can be both positive and negative we use the following lemma:
\begin{lemma}
\label{le:mixed}
Let
$$P(G;X)= \sum_{i=0}^{i=s(G)} h_i(G) X^i$$
be a univariate graph polynomial now with possibly negative integer coefficients.
Then there is a d.p.-equivalent graph polynomial  with non-negative integer coefficients
$$Q(G;X) = \sum_{i=0}^{i=2 \cdot s(G)}  g_i(G) X^{i}.$$
\end{lemma}

\begin{proof}
We define a mapping between the coefficients as follows.
If $h_i(G) \geq 0$ then $g_{2i} =h_i(G)$ and $g_{2i-1} =0$.
If $h_i(G) < 0$ then $g_{2i-1} =|h_i(G)|$ and $g_{2i} =0$.
Clearly, $g_i(G) \geq 0$ for all $i$, and
$g_i(G)$ is computable from all the values of $h_i(G)$.
Conversely, $h_i(G)$ is also computable from the values of $g_i(G)$.
\end{proof}

Combining Lemma \ref{le:mixed}
with Theorem \ref{th:main1} we get:
\begin{theorem}
\label{th:main1a}
Let 
$$P(G;X)= \sum_{i=0}^{i=s(G)} h_i(G) X^i$$
with integer coefficients.
Then there exist two  univariate graph polynomials
$P_1(G;X)$ 
and $P_2(G;X)$ 
which are d.p.-equivalent
to $P(G;X)$ such that
\begin{enumerate}
\item
$P_1(G;X)$ has no negative real roots, and
\item
$P_2(G;X)$ has no real roots besides possibly $0$.
\end{enumerate}
\end{theorem}
\begin{remark}
For those familiar with the notion of definability of graph polynomials in
Second Order Logic $\SOL$ as developed in \cite{ar:KotekMakowskyZilber11},
it is not difficult to see that $Q(G;X), P_1(G;X)$ and $P_2(G;X)$
can be made $\SOL$-definable, provided $P(G;X)$ is $\SOL$-definable.
\end{remark}

Again let $P(G;X)$ be defined as in Equation (\ref{P}),
\begin{gather}
P(G;X)= \sum_{i=0}^{i=s(G)} h_i(G) X^i
\notag
\end{gather}
with $s(G)$ a similarity function with values in $\N$.
We want to find $P_3(G;X)$ d.p.-equivalent to $P(G;X)$, such that $P_3(G;X)$ has only real roots.

%------------new version
A suitable candidate for $P_3(G;X)$ is
\begin{gather}
\label{P3}
P_3(G;X) = \prod_{i=0}^{i= s(G)} (X - i)^{h_i(G)+1} = \sum_{i=0}^{i= s(G)} H_i(G) X^i
\notag
\end{gather}
\begin{remark}
For those familiar with the notion of definability of graph polynomials in
Second Order Logic $\SOL$ as developed in \cite{ar:KotekMakowskyZilber11},
it is not difficult to see that $P_3(G;X)$
can be made $\SOL$-definable, provided $P(G;X)$ is $\SOL$-definable.
One has to code $i$ as in initial segment of the ordered set of  vertices $V(G)$.
\end{remark}

\begin{lemma}
\label{le:realroots}
$$
 P_3(G;X) = \prod_{i=0}^{i= s(G)} (X - i){h_i(G)+1}
$$ 
is d.p.-equivalent to
$$
P(G;X)= \sum_{i=0}^{i= s(G)} h_i(G) X^i
$$ 
and all the roots of $ P_3(G;X)$ are real (even integers).
\end{lemma}
\begin{proof}
Using Lemma \ref{le:mixed} we can assume without loss of generality that the coefficients $h_i(G)$
are non-negative integers.
The proof now reduces to the following observation:
$P(G;X)$ has $i$ as a root with multiplicity $h_i(G)+1$
iff $h_i(G)$ was the coefficient of $X^i$ in $P(G;X)$.
$P_3$ has only non-negative integer, hence real roots.
Finally, we use Proposition \ref{prop:equiv} to show that 
$P$ and $P_3$ are d.p.-equivalent. Given the coefficients $h_i(G)$ of $P(G;X)$ we compute the coefficients
$H_i(G)$ of $P_3(G;X)$ by multiplying out.
Conversely, given he coefficients $H_i(G)$ of $P_3(G;X)$ we compute the roots with their multiplicities
to get the coefficients $h_i(G)$.
\end{proof}
%------------old version
\ifskip
\else
A suitable candidate for $P_3(G;X)$ is
\begin{gather}
\label{P3}
P_3(G;X) = \prod_{i=0}^{i= s(G)} (X - h_i(G)) = \sum_{i=0}^{i= s(G)} H_i(G) X^i
\end{gather}
\begin{remark}
For those familiar with the notion of definability of graph polynomials in
Second Order Logic $\SOL$ as developed in \cite{ar:KotekMakowskyZilber11},
it is not difficult to see that $P_3(G;X)$
can be made $\SOL$-definable, provided $P(G;X)$ is $\SOL$-definable.
\end{remark}

\begin{lemma}
\label{le:realroots}
$$
 P_3(G;X) = \prod_{i=0}^{i= s(G)} (X - h_i(G))
$$ 
is d.p.-equivalent to
$$
P(G;X)= \sum_{i=0}^{i= s(G)} h_i(G) X^i
$$ 
and all the roots of $ P_3(G;X)$ are real (even integers).
\end{lemma}
\begin{proof}
The roots of $P_3(G;X)$ are $h_i(G)$, which are integers by assumption.
\\
Clearly, 
$$P_3(G;X) \preceq_{d.p.} P(G;X).$$ 
We can compute the coefficients $H_i(G)$ from
the coefficients $h_i(G)$ by multiplying out and then use Proposition \ref{prop:equiv}.
\\
To show that 
$$P(G;X) \preceq_{d.p.} P_3(G;X)$$
we compute the roots $z_1(G), \ldots , z_{s(G)}$ of $P_3(G;X)$ with their multiplicities.
The problem is that we do not know the permutation $\pi: [s(G)] \rightarrow [s(G)]$ which would give us
$z_i(G) = h_{\pi(i)}(G)$.

However we do not need this.
\begin{claim}
Let $\Pi$ a function which associates with each graph $G$ a permutation
$\Pi(G): [s(G)] \rightarrow [s(G)]$.
and let
$$P_{\Pi}(G;X)= \sum_{i=0}^{i=s(G)} h_{\Pi(G)(i)}(G) X^i.$$
Then $P(G;X)$ and $P_{\Pi}(G;X)$ are d.p.-equivalent.
\end{claim}
To see this,
we again use Proposition \ref{prop:equiv}.
\end{proof}
\fi %skip
\begin{remark}
To make $P_{\Pi}(G;X)$ definable in $\SOL$ provided $P(G;X)$ is $\SOL$-definable, we need some
additional assumptions about the function $\Pi$. 
%which associates with each graph $G$
%a permutation $\pi_G: [s(G)] \rightarrow [s(G)]$.
\end{remark}
%----------------end old version -----------------------------------------

\begin{theorem}
\label{th:main2}
Let $s(G)$ be a similarity function with values in $\N$.
For every univariate graph polynomial with integer coefficients 
$$P(G;X) = \sum_{i=0}^{i= s(G)} h_i(G) X^i$$
there is a d.p.-equivalent graph polynomial 
$$Q(G;X) = \sum_{i=0}^{i= s(G)} H_i(G) X^i$$
with integer coefficients such that all the roots of $Q(G;X)$ are real.
\end{theorem}
\begin{proof}
Take $Q(G;X) = P_3(G;X)$ from Lemma \ref{le:realroots}. 
%and notice that
%he choice of the permutation $\pi$ only depends on the similarity function $s(G)$.
\end{proof}

%------------------------------------------------------------------------
\subsection{Density}
We first construct a similarity polynomial $D_{\C}(G;X)$ with roots dense in the complex plane $\C$
which will serve as a universal prefactor.
\begin{lemma}
\label{le:dense-R}
There exist  
univariate similarity polynomials 
$D_{\R}^+(G;X)$
$D_{\R}^-(G;X)$
such that all its roots of
$D_{\R}^+(G;X)$ 
($D_{\R}^-(G;X)$) 
are real and dense in 
$[0, \infty) \subseteq \R$
($(-\infty, 0] \subseteq \R$).
\end{lemma}
\begin{proof}
Put
$$D_{\R}^+(G;X)  = (k(G) X - |V(G)|) \cdot (|V(G)| X - k(G))$$
and
$$D_{\R}^-(G;X)  = (k(G) X + |V(G)|) \cdot (|V(G)| X + k(G))$$
%We have seen in Section \ref{se:sol} that $X$, $k(G)$ and $|V(G)|$ are $\SOL$-definable,
%and that $\SOL$-definable are closed under products, sums and substitutions.
%Hence $D_{\R}(G;X)$ is $\SOL$-definable and univariate. 
As $D_{\R}^+(G;X)$ only depends on $|V(G)|$ and $k(G)$
it is a similarity polynomial.
The roots of $D_{\R}^+(G;X)$ are of the form $\frac{|V(G)|}{k(G)}$ or $\frac{k(G)}{|V(G)|}$.
The only limitation for these values is given by $k(G)\leq |V(G)|$ and $k(G) \geq 0$.
So the roots form a dense subset of the rational numbers.
The argument for
$D_{\R}^-(G;X)$ is basically the same.
\end{proof}

The following is straightforward.
\begin{lemma}
\label{le:similar}
Let $G=(V(G),E(G))$ be a graph with $k(G)$ connected components. Then
\begin{enumerate}
\item
$|E(G)| \leq {{|V(G)|- k(G) +1}\choose{2}}$
\item
$|V(G)| - |E(G)| \leq k(G) \leq |V(G)|$
\end{enumerate}
Conversely, if three non-negative integers $v,e,k$ satisfy
\begin{enumerate}
\item
$e \leq {{v- k +1}\choose{2}}$
\item
$v- e \leq k \leq v$
\end{enumerate}
then there exists a graph $G=(V(G),E(G))$ such that
$v= |V(G)|$
$e= |E(G)|$ and $k=k(G)$.
\end{lemma}

\begin{lemma}
\label{le:dense-C}
There exist  
%$\SOL$-definable 
univariate similarity polynomials 
$D_{\C}^i(G;X), i=1,2,3,4$
of degree $12$
such that all the roots of
$D_{\C}^i(G;X)$ 
are dense in  the $i$th quadrant of $\C$.
\end{lemma}
\begin{proof}
We prove the lemma for the first quadrant of $\C$, the other cases being essentially the same.
We first look at the polynomial
$$
D^1(X)= D^1(X,a,b,c,a',b')= (X - \frac{a +b i}{c}) \cdot (X - (a' + b'i))
$$
with $a,b,c,a',b' \in \N$.
We find that for
$a' = \frac{a}{c}$ and $b' = \frac{-b}{c}$   we get
the quadratic polynomial
\begin{gather}
D^1(X,a,b,c)= 
c^2X^2 
- 2acX
+(a^2 +b^2) 
\tag{*}
\end{gather}
with $D^1(X) \in \Z[X]$ and for all positive integers $a,b,c$ the complex numbers
$\frac{a+bi}{c}$ roots roots in the first quadrant of $\C$.
These roots are dense in the first quadrant, even if we assume that the numbers $a,b,c$ are
distinct.
Furthermore, the polynomial remains the same if we replace $(a,b,c)$ by some multiple $(ja, jb, jc)$.
In other words
\begin{gather}
D^1(X,a,b,c) =
D^1(X,ja,jb,jc)
\tag{**}
\end{gather}

We now want to convert $D^1(X,a,b,c)$ into a similarity polynomial by assigning
$|V(G)|$, $|E(G)|$ or $k(G)$ to the parameters $a,b,c$.

Let 
$$\Pi= \{\pi : \pi: \{a,b,c\} \rightarrow \{|V(G)|, |E(G)|, k(G)\} \}$$
and put
\begin{gather}
D_{\C}^1(G;X) = 
\prod_{\pi \in \Pi}  D^1(X, \pi(a), \pi(b), \pi(c))
=
\notag \\
= \prod_{\pi \in \Pi} 
\pi(c)^2X^2 
- 2\pi(a) \pi(c)X
+(\pi(a)^2 +\pi(b)^2) 
\tag{***}
\end{gather}
Clearly, this is a similarity polynomial which is  $\SOL$-definable with roots
$$
\frac{\pi(a) + \pi(b) i}{c}
$$.

By Lemma \ref{le:similar}
We have the following constraints:
\begin{enumerate}
\item
$$|E(G)| \leq {{|V(G)|- k(G) +1}\choose{2}} = \frac{(|V(G)|- k(G) +1)(|V(G)|- k(G))}{2}$$
\item
$$|V(G)| - |E(G)| \leq k(G) \leq |V(G)|$$
\end{enumerate}

We have to show that for every three distinct integers $a,b,c$ there is 
a graph $G$ and $\pi \in \Pi$ such that
$V= |V(G)| =\pi(a), E=|E(G)| = \pi(b)$ and $k=k(G) = \pi(c)$.

Let $\pi_0$ by such that 
$\pi_0(\max\{a,b,c\}) = E$
and
$\pi_0(\min\{a,b,c\}) = k$.
This satisfies constraint (ii).

If there is no graph $G$ for which the constraint (i) is satisfied,
we put
$j= 2E$ in (**) and get a  triple, $E',k',V'$ 
with $E' = E^2, k' = E \cdot k$ and $V' = E\cdot V$,
which satisfies (i). 
To see this note that (i) becomes
\begin{gather}
j2E =j^2 \leq j^2(V-k+\frac{1}{j})(V-k) \notag 
\end{gather}
hence
\begin{gather}
1 = \leq j^2(V-k+\frac{1}{j})(V-k) \notag 
\end{gather}
which is true for $V > k$.
But $V >k$ since $V \geq k$ by (ii), and we have assumed that $V,E, k$ are all distinct.

So we can find a graph
$G'$ with $|E(G')| = E' , k(G') = k'$ and $|V(G')| = V'$.

But by (**) we have not changed the polynomial $D^1(X)$.
\end{proof}

\begin{theorem}
\label{th:main3}
For every univariate graph polynomial $P(G;X)$ there is a univariate graph polynomial $Q(G;X)$
which is prefactor equivalent to $P(G;X)$ and the roots of $Q(G;X)$ are dense in $\C$.
\end{theorem}
\begin{proof}
We use Lemma \ref{le:dense-C} and put
$$
Q(G;X) =  \left(\prod_{i=1}^{i=4} D^i(G;X)\right) \cdot P(G;X).
$$
\end{proof}
\begin{remark}
Note that the polynomials $D^i(G;X)$ are independent of $P(G;X)$
and are easily seen to be $\SOL$-definable. Therefore $\prod_{i=1}^{i=4} D^i(G;X)$ is also $\SOL$-definable,
cf. \cite{ar:KotekMakowskyZilber11}.
\end{remark}

\begin{theorem}
\label{th:main3a}
For every univariate graph polynomial 
$$
P(G;X) = \sum_{i=0}^{i= s(G)} h_i(G) X_i
$$ 
with $s(G)$ a similarity function
there is a univariate graph polynomial $Q(G;X)$
which is d.p.-euqivalent to $P(G;X)$ and the roots of $Q(G;X)$ are all real and dense in $\R$.
\end{theorem}
\begin{proof}
We combine Theorem \ref{th:main2} with the
Lemmas \ref{le:dense-R} and \ref{le:realroots} where $P_3(G;X)$ is d.p.-equivalent to $P(G;X)$
and put
$$
Q(G;X)= D_{\R}^+(G;X) \cdot P_3(G;X)
$$
\end{proof}

%------------------------------------------------------------------------
\subsection{Bounding complex roots in a disk}
To get a similar theorem bounding the complex roots in a disk we use
Rouch\'e's Theorem, cf.  \cite[Section 4.10, Theorem 4.10c]{bk:Henrici-vol1}.

\begin{theorem}[Rouch\'e's Theorem]
\label{th:rouche}
Let 
$P(X) = \sum_{i=0}^{d} h_i X^i$
be a polynomial and
\begin{gather}
R = 1 + \frac{1}{|h_d|} \cdot \max_i \{| h_i |: 0 \leq i \leq d-1\}
\notag
\end{gather}
Then all complex roots $\xi$ of $P(X)$ satisfy $|\xi| \leq R$.
\end{theorem}

We shall use Rouch\'e's Theorem in the form given by the following corollary:
\begin{corollary}
\label{cor:rouche}
Let
$P(X) = \sum_{i=0}^{d} h_i X^i$ be a polynomial with integer coefficients and $h_d \geq 1$,
and let
$g(X) = A \cdot X$
with $A \geq \max_i \{| h_i |: 0 \leq i \leq d-1\}$. Define
$$
P'(X) = P(A\cdot X) = \sum_{i=0}^{d} h_i A^i X^i = \sum_{i=0}^{d} H_i X^i
$$ 
with $H_i = A^i h_i$.
Then all complex roots $\xi$ of $P'(X)$ satisfy $|\xi| \leq 2$.
\end{corollary}
\begin{proof}
If all $h_i: i=0, \ldots, d-1$ vanish, $P(X)= h_d X^d$ and $0$ is the only root and has multiplicity $d$.
Therefore, without loss of generality,
we can assume that $A \geq 1$, because all $h_i$ are integers.
We have to show that the coefficients $H_i$ satisfy the hypotheses of Theorem \ref{th:rouche} with $R=2$.
$$
2= 1+ \frac{1}{ A^d \cdot |h_d|} \cdot \max_i \{ A^i \cdot | h_i |: 0 \leq i \leq d-1\}
$$
it suffices to show that for $i \leq d-1$ we have
$$
\frac{A^i \cdot h_i}{A^d \cdot h_d} \leq A^d \cdot |h_d|
$$
If $h_i=0$ this is true.
If $h_i \neq 0$ we have 
%that $\frac{h_i}{h_d} \geq \frac{1}{h_d}$ and therefore
$$
\frac{A^i \cdot h_i}{A^d \cdot h_d} \leq 
\frac{A^{i+1}}{A^d \cdot h_d} \leq 
A^d \cdot |h_d|
$$
because
$$
\frac{A^{i+1}}{A^d \cdot h_d} \leq  1
$$
and
$$
1 \leq A^d \cdot |h_d|
$$
\end{proof}

\begin{theorem}
\label{th:main4}
Let $P(G;X) = \sum_{i=0}^{d(G)} h_i(G) X^i$ be a univariate graph polynomial with integer coefficients,
and such that $|h_i(G)| \leq |V(G)|^r$ for some fixed $r \in \N$.
Then there exists a univariate graph polynomial
$P_1(G;X)$ which is 
substitution equivalent
%(d.p)-equivalent 
to $P(G;X)$ such that
all complex roots $\xi$ of $P_1(G;X)$ satisfy $|\xi| \leq 2$.
\end{theorem}
\begin{proof}
We use $g(G;X) = n(G)^r \cdot X$ from Corollary \ref{cor:rouche} and substitute it for $X$.
Clearly $g(G;X)$ is a similarity polynomial, and its inverse is a rational similarity function.
\end{proof}

Combining Theorem \ref{th:main4} with Theorem \ref{th:main2} we get

\begin{corollary}
\label{cor:realbounded}
Let $s(G)$ be a similarity function.
For every univariate graph polynomial with real coefficients 
$$P(G;X) = \sum_{i=0}^{i= s(G)} h_i(G) X_i$$
there is a d.p.-equivalent graph polynomial 
$$Q(G;X) = \sum_{i=0}^{i= s(G)} H_i(G) X_i$$
with real coefficients such that all the roots $z$ of $Q(G;X)$ are real and $|z| \leq 2$.
\end{corollary}

\begin{remark}
Can we replace d.p.-equivalence by prefactor equivalence in Theorem \ref{th:main2} and Corollary \ref{cor:realbounded}?
\end{remark}

%---------------Conclusions-------------------
\section{Conclusions and open problems}
\label{conclu}
\label{se:conclu}

We have formalized several notions of reducibility and equivalence of graph polynomials
which are implicitly used in the literature: d.p.-equivalence, prefactor equivalence and
substitution equivalence. We used these notions to discuss whether the locations
of the roots of a univariate graph polynomial $P(G;X)$ are meaningful.
We have shown that, under some weak assumptions,  there is always a d.p.-equivalent
polynomial $Q(G;X)$ such that its roots are always real and dense (or bounded) in $\R$.
We have shown that, under some weak assumptions, there is always a prefactor equivalent polynomial
$Q(G;X)$ such that its roots are dense (or bounded) in $\C$.

As our results show, d.p.-equivalence allows rather dramatic modifications of the presentation of
graph polynomials. This is also the case for prefactor equivalence, although in a less dramatic way.
In the case of d.p.equivalence we could require that the transformation of the
coefficients in Proposition \ref{prop:equiv} be restricted to transformations of low algebraic or
computational complexity, but one would like that the various representations of a graph polynomial
from Example \ref{ex:coefficients} remain equivalent.
We did not address such refinements of d.p.-equivalence in this paper, and this option for further
research.

Ultimately, we are faced with the question:
\begin{quote}
What makes a graph polynomial interesting within its 
\\
d.p.- equivalence or prefactor equivalence class?
\end{quote}

To avoid {\em unnatural} graph polynomials we might require that the coefficients
have a combinatorial interpretation.
This can be captured be requiring that the graph polynomial be definable in Second Order Logic
$\SOL$, as proposed in \cite{ar:KotekMakowskyZilber11}.
However, this is much too general and our modification theorems for the location
of the roots still apply under such a restriction.

A graph polynomial is called an {\em elimination invariant} 
if it satisfies
some recurrence relation with respect
to certain vertex and/or edge elimination operations.
In the case of the Tutte polynomial one speaks of a {\em Tutte-Grothendieck invariant}
(TG-invariant), in the case of the 
chromatic and dichromatic polynomial of a {\em chromatic invariant} (C-invariant), 
cf. \cite{bk:Bollobas98,bk:Aigner2007,pr:Ellis-MonaghanMerino2008,pr:Ellis-MonaghanMerino2008a}.
Other cases are the M-invariants for matching polynomials, and the EE-invariants
and VE-invariants from \cite{phd:Averbouch}, cf. also
\cite{ar:AverbouchGodlinMakowsky10,ar:TittmannAverbouchMakowsky10}.

Several graph polynomials $U_E(G; \bar{X})$ have been characterized as the most general
elimination invariant of a certain type $E \in \{C,TG,M, VE, EE\}$ in the following sense:
\begin{enumerate}
\item
$U_E(G; \bar{X})$ is an E-invariant;
\item
every other E-invariant is a substitution instance of $U_E(G; \bar{X})$.
\item
A well known E-invariant, say $P(G; \bar{X})$, is prefactor equivalent to $U_E(G; \bar{X})$.
\end{enumerate}
Theorems of this form are also called {\em recipe theorems}, cf. \cite{bk:Aigner2007,ar:Ellis-MonaghanSarmiento2011}. 
However, in such cases the location of the zeros of a univariate E-invariant $P(G;\bar{X})$  are still subject
to our modification theorems.

Other graph polynomials in the literature were obtained from counting weighted homomorphisms,
cf. \cite{bk:Lovasz-hom} or the recent \cite{ar:GarijoGoodallNesetril2013}, or from counting generalized
colorings, \cite{ar:delaHarpeJaeger1995,ar:KotekMakowskyZilber11}.
These frameworks are quite general and are unlikely avoid our modification theorems.
It remains a challenge to define a framework in which the location of the roots of a
graph polynomial is semantically significant.

%---------------Bibliography-------------------
%\bibliographystyle{plain}
%\bibliography{/home/janos/MyLatex/texbib/acer,/home/janos/MyLatex/texbib/books,/home/janos/MyLatex/texbib/unpub,/home/janos/MyLatex/texbib/phd,/home/janos/MyLatex/texbib/proc,/home/janos/MyLatex/texbib/new0,/home/janos/MyLatex/texbib/anew,/home/janos/MyLatex/texbib/graphs,/home/janos/MyLatex/texbib/perman,/home/janos/MyLatex/texbib/aachen,/home/janos/MyLatex/texbib/articles,/home/janos/MyLatex/texbib/db}

\begin{thebibliography}{10}

\bibitem{bk:Aigner2007}
M.~Aigner.
\newblock {\em A course in enumeration}.
\newblock Graduate Texts in Mathematics. Springer, 2007.

\bibitem{ar:AkbariAlikhaniPeng2010}
S.~Akbari, S.~Alikhani, and Y.-H. Peng.
\newblock Characterization of graphs using domination polynomials.
\newblock {\em Eur. J. Comb.}, 31(7):1714--1724, 2010.

\bibitem{ar:AkbariOboudi2013}
S.~Akbari and M.~R. Oboudi.
\newblock On the edge cover polynomial of a graph.
\newblock {\em Eur. J. Comb.}, 34(2):297--321, 2013.

\bibitem{phd:Alikhani}
S.~Alikhani.
\newblock {\em Dominating Sets and Domination Polynomials of Graphs}.
\newblock PhD thesis, Universiti Putra Malaysia, 2009.
\newblock http://psasir.upm.edu.my/7250/.

\bibitem{ar:ArochaLlano2000}
J.~L. Arocha and B.~Llano.
\newblock Mean value for the matching and dominating polynomial.
\newblock {\em Discussiones Mathematicae Graph Theory}, 20(1):57--69, 2000.

\bibitem{phd:Averbouch}
I.~Averbouch.
\newblock {\em Completeness and Universality Properties of Graph Invariants and
  Graph Polynomials}.
\newblock PhD thesis, Technion - Israel Institute of Technology, Haifa, Israel,
  2011.

\bibitem{ar:AverbouchGodlinMakowsky10}
I.~Averbouch, B.~Godlin, and J.A. Makowsky.
\newblock An extension of the bivariate chromatic polynomial.
\newblock {\em European Journal of Combinatorics}, 31(1):1--17, 2010.

\bibitem{ar:Balaban93}
A.~T. Balaban.
\newblock Solved and unsolved problems in chemical graph theory. quo vadis,
  graph theory?
\newblock {\em Ann. Discrete Math.}, 35:109--126, 1993.

\bibitem{ar:Balaban95}
A.~T. Balaban.
\newblock Chemical graphs: Looking back and glimpsing ahead.
\newblock {\em Journal of Chemical Information and Computer Science},
  35:339--350, 1995.

\bibitem{bk:Bollobas98}
B.~Bollob\'as.
\newblock {\em Modern Graph Theory}.
\newblock Springer, 1998.

\bibitem{ar:BrentiRoyleWagner1994}
F.~Brenti, G.F. Royle, and D.G. Wagner.
\newblock Location of zeros of chromatic and related polynomials of graphs.
\newblock {\em Canadian Journal of Mathematics}, 46:55--80, 1994.

\bibitem{bk:BrouwerHaemers2012}
A.E. Brouwer and W.H. Haemers.
\newblock {\em Spectra of Graphs}.
\newblock Springer Universitext. Springer, 2012.

\bibitem{ar:BrownHickmanNowakowski2004}
J.I. Brown, C.A. Hickman, and R.J. Nowakowski.
\newblock On the location of roots of independence polynomials.
\newblock {\em Journal of Algebraic Combinatorics}, 19.3:273--282, 2004.

\bibitem{ar:BrownTufts2013}
J.I. Brown and J.~Tufts.
\newblock On the roots of domination polynomials.
\newblock {\em Graphs and Combinatorics}, published online:DOI
  10.1007/s00373--013--1306--z, 2013.

\bibitem{ar:CsikvariOboudi2011}
P.~Csikv{\'a}ri and M.~R. Oboudi.
\newblock On the roots of edge cover polynomials of graphs.
\newblock {\em Eur. J. Comb.}, 32(8):1407--1416, 2011.

\bibitem{bk:CvetkovicDoobSachs1995}
D.M. Cvetkovi{\'c}, M.~Doob, and H.~Sachs.
\newblock {\em Spectra of {G}raphs}.
\newblock Johann Ambrosius Barth, 3rd edition, 1995.

\bibitem{ar:delaHarpeJaeger1995}
P.~de~la Harpe and F.~Jaeger.
\newblock Chromatic invariants for finite graphs: Theme and polynomial
  variations.
\newblock {\em Linear Algebra and its Applications}, 226-228:687--722, 1995.

\bibitem{ar:DongHendyTeoLittle02}
F.M. Dong, M.D. Hendy, K.L. Teo, and C.H.C. Little.
\newblock The vertex-cover polynomial of a graph.
\newblock {\em Discrete Mathematics}, 250:71--78, 2002.

\bibitem{bk:DongKohTeo2005}
F.M. Dong, K.M. Koh, and K.L. Teo.
\newblock {\em Chromatic {P}olynomials and {C}hromaticity of {G}raphs}.
\newblock World Scientific, 2005.

\bibitem{pr:Ellis-MonaghanMerino2008}
J.~Ellis-Monaghan and C.~Merino.
\newblock Graph polynomials and their applications i: The {T}utte polynomial.
\newblock arXiv 0803.3079v2 [math.CO], 2008.

\bibitem{pr:Ellis-MonaghanMerino2008a}
J.~Ellis-Monaghan and C.~Merino.
\newblock Graph polynomials and their applications ii: Interrelations and
  interpretations.
\newblock arXiv 0806.4699v1 [math.CO], 2008.

\bibitem{ar:Ellis-MonaghanSarmiento2011}
J.~A. Ellis-Monaghan and I.~Sarmiento.
\newblock A recipe theorem for the topological {T}utte polynomial of
  {B}ollob{\'a}s and {R}iordan.
\newblock {\em Eur. J. Comb.}, 32(6):782--794, 2011.

\bibitem{ar:GarijoGoodallNesetril2013}
D.~Garijo, A.~Goodall, and J.~{Ne\v{s}et\v{r}il}.
\newblock Polynomial graph invariants from homomorphism numbers.
\newblock arXiv:1308.3999 [math.CO], 2013.

\bibitem{ar:GodlinKatzMakowsky12}
B.~Godlin, E.~Katz, and J.A. Makowsky.
\newblock Graph polynomials: From recursive definitions to subset expansion
  formulas.
\newblock {\em Journal of Logic and Computation}, 22(2):237--265, 2012.

\bibitem{ar:GoldwurmSantini2000}
M.~Goldwurm and M.~Santini.
\newblock Clique polynomials have a unique root of smallest modulus.
\newblock {\em Inf. Process. Lett.}, 75(3):127--132, 2000.

\bibitem{ar:GutmanHarary83}
I.~Gutman and F.~Harary.
\newblock Generalizations of the matching polynomial.
\newblock {\em Utilitas Mathematicae}, 24:97--106, 1983.

\bibitem{ar:HeilmannLieb72}
C.J. Heilmann and E.H. Lieb.
\newblock Theory of monomer-dymer systems.
\newblock {\em Comm. Math. Phys}, 25:190--232, 1972.

\bibitem{bk:Henrici-vol1}
P.~Henrici.
\newblock {\em Applied and Computational Complex Analysis, volume 1}.
\newblock Wiley Classics Library. John Wiley, 1988.

\bibitem{ar:HoedeLi94}
C.~Hoede and X.~Li.
\newblock Clique polynomials and independent set polynomials of graphs.
\newblock {\em Discrete Mathematics}, 125:219--228, 1994.

\bibitem{phd:Hoshino}
R.~Hoshino.
\newblock {\em Independence Polynomials of Circulant Graphs}.
\newblock PhD thesis, Dalhousie University, Halifax, Nova Scotia, 2007.

\bibitem{ar:KotekMakowskyZilber11}
T.~Kotek, J.A. Makowsky, and B.~Zilber.
\newblock On counting generalized colorings.
\newblock In M.~Grohe and J.A. Makowsky, editors, {\em Model Theoretic Methods
  in Finite Combinatorics}, volume 558 of {\em Contemporary Mathematics}, pages
  207--242. American Mathematical Society, 2011.

\bibitem{ar:KotekPreenSimonTittmanTrinks2012}
T.~Kotek, J.~Preen, F.~Simon, P.~Tittmann, and M.~Trinks.
\newblock Recurrence relations and splitting formulas for the domination
  polynomial.
\newblock {\em Electr. J. Comb.}, 19(3):P47, 2012.

\bibitem{pr:LevitMandrescu05}
V.E. Levit and E.~Mandrescu.
\newblock The independence polynomial of a graph - a survey.
\newblock In S.~Bozapalidis, A.~Kalampakas, and G.~Rahonis, editors, {\em
  Proceedings of the 1st International Conference on Algebraic Informatics},
  pages 233--254. Aristotle University of Thessaloniki, Department of
  Mathematics, Thessaloniki, 2005.

\bibitem{bk:Lovasz-hom}
L.Lov\'asz.
\newblock {\em Large Networks and Graph Limits}, volume~60 of {\em Colloquium
  Publications}.
\newblock AMS, 2012.

\bibitem{bk:LovaszPlummer86}
L.~Lov{\'a}sz and M.D. Plummer.
\newblock {\em Matching Theory}, volume~29 of {\em Annals of Discrete
  Mathematics}.
\newblock North Holland, 1986.

\bibitem{ar:MakowskyTARSKI}
J.A. Makowsky.
\newblock Algorithmic uses of the {F}eferman-{V}aught theorem.
\newblock {\em Annals of Pure and Applied Logic}, 126.1-3:159--213, 2004.

\bibitem{up:Lecture-11}
J.A. Makowsky and E.V. Ravve.
\newblock Logical methods in combinatorics, {L}ecture 11.
\newblock Course given in 2009 under the number 236605, Advanced Topics,
  available at: {\tt
  http://www.cs.technion.ac.il/~janos/COURSES/236605-09/lec11.pdf}.

\bibitem{pr:MakowskyRavveErdos}
J.A. Makowsky and E.V. Ravve.
\newblock On the location of roots of graph polynomials.
\newblock {\em Electronic Notes in Discrete Mathematics}, 43:201--206, 2013.

\bibitem{ar:MerinoNoble2009}
C. Merino and S.~D. Noble.
\newblock The equivalence of two graph polynomials and a symmetric function.
\newblock {\em Combinatorics, Probability {\&} Computing}, 18(4):601--615,
  2009.

\bibitem{ar:Sokal01}
A.~D. Sokal.
\newblock Bounds on the complex zeros of (di)chromatic polynomials and
  {P}otts-model partition functions.
\newblock {\em Combinatorics, Probability {\&} Computing}, 10(1):41--77, 2001.

\bibitem{ar:Sokal04}
A.~D. Sokal.
\newblock Chromatic roots are dense in the whole complex plane.
\newblock {\em Combinatorics, Probability {\&} Computing}, 13(2):221--261,
  2004.

\bibitem{ar:Sokal2005a}
A.~D. Sokal.
\newblock The multivariate {T}utte polynomial (alias {P}otts model) for graphs
  and matroids.
\newblock In {\em Survey in Combinatorics, 2005}, volume 327 of {\em London
  Mathematical Society Lecture Notes}, pages 173--226, 2005.

\bibitem{ar:TittmannAverbouchMakowsky10}
P.~Tittmann, I.~Averbouch, and J.A. Makowsky.
\newblock The enumeration of vertex induced subgraphs with respect to the
  number of components.
\newblock {\em European Journal of Combinatorics}, 32(7):954--974, 2011.

\bibitem{bk:Trinajstic1992}
N.~Trinajsti{\'c}.
\newblock {\em Chemical {G}raph {T}heory}.
\newblock CRC Press, 2nd edition, 1992.

\end{thebibliography}

%---------------------------------------------
\end{document}